\documentclass{amsart}
\usepackage{amssymb,stix,bbm}
\usepackage[matrix,arrow,curve]{xy}
\usepackage{verbatim}
\usepackage{graphics}
\usepackage{a4wide,array}
\usepackage[utf8]{inputenc}
\usepackage{tikz}
\usepackage{tikz-cd} 

\newtheorem{theorem}{Theorem}[section]

\newtheorem{proposition}[theorem]{Proposition}

\usepackage{hyperref}
\hypersetup{
    colorlinks,
    citecolor=black,
    filecolor=black,
    linkcolor=black,
    urlcolor=black
}

\def\GL{\mathrm{GL}}

\def\C{\mathbf{C}}
\def\Z{\mathbf{Z}}

\def\kk{\mathbf{k}}

\def\Q{\mathbf{Q}}
\def\C{\mathbf{C}}

\def\Id{\mathrm{Id}}

\def\onto{\twoheadrightarrow}

\def\met1{\mathbf{Met_1}}
\def\pmet1{\mathbf{PMet_1}}

\def\smet1{\mathbf{sMet_1}}
\def\spmet1{\mathbf{sPMet_1}}

\def\nabla{\triangledown}

\def\gr{\mathrm{gr}\,}
\DeclareMathOperator{\Lie}{Lie}

\date{March 25, 2023}
\begin{document}
\centerline{}

\title{The Malcev completion of complex braid groups}
\author[I.~Marin]{Ivan Marin}
\address{LAMFA, UMR CNRS 7352, Universit\'e de Picardie-Jules Verne, Amiens, France, and IMJ-PRG, UMR CNRS 7586, Paris.}
\email{ivan.marin@u-picardie.fr}
\email{marin@imj-prg.fr}

\medskip

\begin{abstract} In this short note we provide an alternative proof of a theorem of Kapovich and Millson about the Malcev completion of an arbitrary Artin group, and determine the
Malcev completion of the braid group of an irreducible finite complex
reflection group.
\end{abstract}

\maketitle

\tableofcontents

\section{Introduction and main results}

We refer the reader to \cite{VKBOOK} Ch. 12 for a modern account of the definition of the Malcev completion $G \otimes \Q$
of a group $G$
and its main properties, some of which will be recalled in Section \ref{sect:prelims}.

Let $S$ be a finite set of vertices and $\Gamma$ be a labelled graph on $S$, determined by
a symmetric Coxeter matrix $(m_{s,t})_{s,t\in S}$ with coefficients in $\{2,3,\dots,\infty\}$.
The Artin group $B = A(\Gamma)$ attached to it is defined by the presentation
with generators $S$ and relations
$$
\underbrace{sts\dots}_{m_{st}}
= 
\underbrace{tst\dots}_{m_{st}}
$$
for $s,t \in S$. It admits a natural epimorphism $A(\Gamma)\onto W(\Gamma)$
with $W(\Gamma)$ the Coxeter group attached to the same Coxeter matrix.

Evidence are given in \cite{RESNIL} that the Kernel $P(\Gamma)$ of $A(\Gamma) \onto W(\Gamma)$
should be residually torsion-free nilpotent, which is equivalent to saying
that it embeds into its Maltev completion $P(\Gamma) \otimes \Q$. More
precisely, it is shown in \cite{RESNIL} that, if the Paris representation of $A(\Gamma)$
is faithful -- which is the case in a number of cases -- then $P(\Gamma)$
is residually torsion-free nilpotent. By contrast, 
the Malcev completion
of $A(\Gamma)$ is quite poor. It has been determined by Kapovich and Millson in \cite{KAPO}. This
quite early reference has been communicated to us by A. Suciu, after we found our proof independently.
Since our proof is fairly direct and uses somewhat different tools than the original
one, we present it in this note. Some of its stages will moreover be used in order to extend these results to the braid groups of (non-real)
reflection groups.

We first start by a limited statement. The Artin group $A(\Gamma)$
is said to be \emph{free of infinity} if $m_{st} \neq \infty$ for every $s,t\in S$,
and $G^{ab}$ denotes the abelianization of the group $G$.

\begin{theorem} \label{theo1}
Let $B = A(\Gamma)$ be an Artin group which is free of infinity. Then the natural morphism $B \otimes\Q \onto B^{ab} \otimes \Q$ is an isomorphism.
\end{theorem}

The abelianization of $B$ is easy to compute from the presentation itself : abelianizing
the presentation makes all relations corresponding to even labels vanish, and the remaining ones
identify the generators involved. Combinatorially, this can be described as follows.
 Considering the
coarsest partition (that is, with the smallest number of parts) of
$S$ as $\bigsqcup_{i=1}^r S_i$ such that, for every two $s_i \in S_i, s_j \in S_j$ with $i \neq j$
we have that $m_{s_i,s_j}$ is even, we get that every two $s,t$ in the same $S_i$ can be
connected by a path with odd labels, so that they are conjugates inside $B$. From this one
readily gets that $B^{ab} \simeq \Z^r$. The above result thus implies $B\otimes \Q \simeq \Q^r$

In the general case, we build a new Artin group attached to a Coxeter graph
$\overline{\Gamma}$ with vertices the $S_i,i=1,\dots,r$ and $\bar{m}_{i,j} = m_{S_i,S_j} = 2$
if we have $m_{s,t} < \infty$ for some $s \in S_i$, $t \in S_j$, and $\bar{m}_{i,j} = \infty$
otherwise. Then $A(\overline{\Gamma})$ is a right-angled Artin group (RAAG).

We have a surjective homomorphism $A(\Gamma) \to A(\overline{\Gamma})$ with maps
each $s \in S_i$ to the generator $S_i$ of $A(\overline{\Gamma})$. Indeed, if $s \in S_i$
and $t \in S_j$ satisfy $i \neq j$, then either $m_{s,t} = 2m$ for some integer $m$,
in which case $S_i S_j = S_j S_i$ and $(st)^m$ and $(ts)^m$ are both mapped
to $(S_i S_j)^m= (S_j S_i)^m$, or $m_{s,t} =\infty$ and there is nothing to check.
This induces an homomorphism $A(\Gamma) \otimes \Q \to A(\overline{\Gamma}) \otimes \Q$.
The full result of Kapovich and Millson is then the following one.

\begin{theorem}
\label{theo1b}
Let $\Gamma$ be an arbitrary Coxeter graph. Then the morphism $A(\Gamma) \otimes \Q \to A(\overline{\Gamma}) \otimes \Q$ is an isomorphism.
\end{theorem}

This completes the task of determinating the Malcev completion of Artin groups,
as the case of a RAAG is simple enough (see Proposition \ref{prop:raags} below).

Our original result then concerns the generalized braid groups attached to an arbitrary
finite complex reflection group, that is a finite subgroup $W$ of $\GL_n(\C)$
generated by complex (pseudo-)reflections. Its braid group is defined as $B = \pi_1(X/W)$,
where $X$ is the complement inside $\C^n$ of the hyperplane arrangement $\mathcal{A}$
made of the fixed point sets of the reflections. We refer to \cite{BMR} for basic results on
these groups.

 When $W$ is irreducible, it is an easy consequence of the classification of complex
 reflection groups that the number $c(W) = |\mathcal{A}/W|$ of orbits of hyperplanes
 under the natural action of $W$ is at most $3$, and that it can be equal to 3 only in rank $n = 2$. Then, one has $B^{ab} \simeq \Z^{c(W)}$ (see \cite{BMR} Theorem 2.17).
 
 The result is the following one, where $F_2$ is the free group on 2 generators.
\begin{theorem} \label{theo2}
Let $W$ be an irreducible complex reflection group. Then the natural morphism
$B\otimes \Q \onto B^{ab} \otimes \Q \simeq \Q^{c(W)}$ is an isomorphism except
if $c(W) = 3$. In this case, we have $B \simeq \Z \times F_2$ whence
$B$ is residually torsion-free nilpotent and
$B \otimes \Q \simeq \Q \times (F_2\otimes \Q)$
\end{theorem}

\section{Preliminaries on the Malcev completion}
\label{sect:prelims}

Let $G$ be a group. For $x,y \in G$ we set $(x,y) = xyx^{-1}y^{-1}$ and,
for $H_1,H_2< G$ two subgroups of $G$ we denote $(H_1,H_2)$ the subgroup generated
by the $(x,y)$ for $x \in H_1$, $y \in H_2$. The lower central series is defined by the sequence $C^1 G = G$, $C^{n+1} G = (G,C^n G)$. The commutator map $G \times G  \to G$
given by $(x,y)$ induces a Lie algebra structure on the graded
$\Z$-module $\gr G = \bigoplus_{n=1}^{\infty} C^n G/C^{n+1} G$. One of its
main properties is that it is generated as a Lie algebra by $\gr_1 G = G/(G,G) = G^{ab}$
(see \cite{BOURBLIE23} ch. 2 \S 4).

The lower central series without torsion is defined by
$TC^n G = \{ g \in G \ | \ \exists m \neq 0 \ g^m \in C^n G\}$. Let $I$ be the
augmentation ideal of the group algebra $\Q G$, that is the kernel of the
augmentation map $\Q G \to \Q$ mapping each $g \in G$ to $1$. Then $TC^n G$
is equal to the kernel of the natural map $G \to \Q G/I^{n+1}$ (see \cite{JENNINGS,QUILLENGRADED})
and the Malcev completion of $G$ is $G \otimes \Q = \varprojlim G/TC^n G \otimes \Q$ where
$G/TC^n G \otimes \Q$ is the original Malcev completion of the torsion-free nilpotent group
$G/TC^n G$ as in \cite{MALTSEVORIG}.

We consider the case of so-called right-angled Artin groups (RAAG), namely the
case where $m_{s,t} \in \{ 2, \infty\}$ for each $s,t \in S$. In this case,
we define $\mathcal{X}(\Gamma)$ to be the graded Lie algebra over $\Q$ with generators
$x_s, \in S$ and relations $[x_s,x_t]= 0$ if $m_{s,t}=2$, and denote $\widehat{\mathcal{X}}(\Gamma)$ its completion with respect to the grading. It is easy to see that
the envelopping algebra of $\mathcal{X}(\Gamma)$ can be identified with the
(graded) unital associative algebra $\mathcal{A}(\Gamma)$
with generators the $x_s, s \in S$ and
relations $x_sx_t=x_tx_s$ when $m_{s,t}=2$ (which is actually the monoid algebra
of the corresponding Artin monoid). The latter is therefore a Hopf algebra with coproduct $\Delta(x_s) = x_s \otimes 1 + 1 \otimes x_s$ for $s \in S$, and $\exp \widehat{\mathcal{X}}(\Gamma)$ can be identified with the grouplike elements of $\widehat{\mathcal{A}}(\Gamma)$
(see \cite{QUILLENRATHOM}) where $\widehat{\mathcal{A}}(\Gamma)$ is the completion of $\mathcal{A}(\Gamma)$ w.r.t. the grading. A similar statement can be found in \cite{KAPO}.

\begin{proposition}\label{prop:raags} Let $A(\Gamma)$ be a RAAG. Then $A(\Gamma) \otimes \Q \simeq \exp \widehat{\mathcal{X}}(\Gamma)$.
\end{proposition}
\begin{proof}
From the presentation of $A(\Gamma)$ we get that there is a well-defined morphism
mapping each $s \in S$ to $\exp(x_s)$, as $st=ts \Rightarrow x_s = x_t \Rightarrow\exp(x_s) \exp(x_t) = \exp(x_t)\exp(x_s) = \exp(x_s+x_t)$. The augmentation ideal of $A(\Gamma)$ is mapped
to elements of valuation at least $1$ inside $\widehat{\mathcal{A}(\Gamma)}$,
so that this morphism extends to a morphism $A(\Gamma) \otimes \Q \to \exp \widehat{\mathcal{X}}(\Gamma)$. Conversely, we can define similarly a Lie algebra morphism $\mathcal{X}(\Gamma)
\to \Lie(A(\Gamma) \otimes \Q)$, where $A(\Gamma) \otimes \Q$ is endowed with a structure
of pro-unipotent group as in e.g. \cite{QUILLENGRADED} and $ \Lie(A(\Gamma) \otimes \Q) = \log(A(\Gamma) \otimes \Q)$
is its Lie algebra. Indeed, when $m_{s,t} =2$ we have $st=ts$ hence
$\log(s)\log(t) = \log(t)\log(s)$ inside $\Lie(A(\Gamma) \otimes \Q)$, so that
mapping each $x_s$ to $\log(s)$ for $s \in S$ defines a Lie algebra morphism $\mathcal{X}(\Gamma)
\to \Lie(A(\Gamma) \otimes \Q)$, which can be extended to its completion
$\widehat{\mathcal{X}}(\Gamma)
\to \Lie(A(\Gamma) \otimes \Q)$. This provides a group homomorphism
$\exp \widehat{\mathcal{X}}(\Gamma)
\to A(\Gamma) \otimes \Q)$.

We want to prove that these provide converse isomorphisms. This is equivalent to
considering the corresponding morphisms of Lie algebras between $\Lie A(\Gamma) \otimes \Q$
and $\widehat{\mathcal{X}(\Gamma)}$. One then proves that the composed maps are automorphisms
of $\Lie A(\Gamma) \otimes \Q$
and $\widehat{\mathcal{X}(\Gamma)}$, respectively. But in order to check this it is enough
to check that we get the induced graded morphisms are
isomorphisms of $\gr \Lie A(\Gamma) \otimes \Q \simeq (\gr A(\Gamma)) \otimes \Q$
and $\gr \widehat{\mathcal{X}(\Gamma)} \simeq \mathcal{X}(\Gamma)$, respectively.
But since both graded Lie algebras are generated by their homogeneous
components of degree $1$, it is enough to check that these morphisms are
the identity in degree $1$. This is immediate on each generator, and this concludes the proof.
\end{proof}

\section{Lower central series of dihedral Artin groups}

We consider the case $S = \{ a_0, a_1 \}$
and assume that $e = m_{a_0,a_1}$ is even. Then $B = \langle a_0,a_1 \ | \ (a_0a_1)^{e/2}=(a_1a_0)^{e/2} \rangle$ and  we can already notice that
$B^{ab} = \langle a_0,a_1 \ | \ a_0a_1 = a_1 a_0 \rangle$ is also a dihedral Artin group with $e = 2$. 
The goal of this Section is to prove the following Proposition.

\begin{proposition} \label{prop:dihedral} Let $B$ be an Artin group of dihedral type
$I_2(e)$ with $e$ even. Then $C^2 B / C^3 B \simeq \Z/(e/2)\Z$.

\end{proposition}

It is probably possible to prove this proposition by a direct group-theoretic argument. We
prefer an homological approach.
We consider the Dehornoy-Lafont Order Complex $(D_{\bullet} B, \partial_{\bullet})$ for these groups
attached to the corresponding Artin monoids, with ordering $a_0<a_1$ on the atoms.
We refer to \cite{DEHORNOYLAFONT} for its definition. We have by construction $D_k B = 0$ for $k > 2$ and
$D_0 B$, $D_1 B$ and $D_2 B$ are free $\Z B$-modules with bases $\{ [\emptyset] \}$,
$\{ [a_0],[a_1] \}$ and $\{ [a_0,a_1] \}$, respectively. From the description in \cite{DEHORNOYLAFONT}
it is immediate that $\partial_1([a_i]) = (a_i - 1)[\emptyset]$, $\partial_0([\emptyset]) = 1$
and it is straightforward to show that
$$
\begin{array}{lcl}
\partial_2([a_0,a_1]) &=& \underbrace{a_0a_1\dots a_0}_{e-1}[a_1] - \underbrace{a_1a_0\dots a_1}_{e-1}[a_0] \\
&-& [a_1] -a_1[a_0] - a_1a_0[a_1] - \dots - \underbrace{a_1a_0\dots a_0}_{e-2}[a_1]\\
&+& [a_0] +a_0[a_1] + a_0a_1[a_0] + \dots + \underbrace{a_0a_1\dots a_1}_{e-2}[a_0]\\
\end{array}
$$
In particular, for $e=2$ we have $\partial_2([a_0,a_1]) = (a_0-1)[a_1]-(a_1-1)[a_0]$.

It is already known that $H_2(B,\Z) \simeq \Z$ for every even $e \geq 2$ (see \cite{SALVETTI94}), with
basis the class of $[a_0,a_1]$. We wish
to compute the morphism $H_2(B,\Z) \to H_2(B^{ab},\Z)$ induced by $\varphi$. For this we use
the standard method of see e.g. \cite{BROWN} p. 48 to consider the acyclic complex $(D_{\bullet}(B^{ab}),\partial^{ab})$
as a complex of $\Z M$-modules via the morphism $\varphi :B \to B^{ab}$, and
construct a morphism of $\Z M$-complexes $f : D_{\bullet}(B) \to D_{\bullet}(B^{ab})$.
Since $D_{\bullet}(B)$ is a complex of projective modules and $D_{\bullet}(B^{ab})$
is acyclic one knows that such a morphism exists and is unique up to homotopy.
Since each $f_i$ is a morphism of $\Z B$-modules one needs to specify only its
values on the chosen basis of $D_i(M)$.
One takes obvisouly $f_0([\emptyset]) =  [\emptyset]$
and $f_1([a_i]) = [a_i]$. Then, one needs to find $x \in \Z B^{ab}$ such that setting
$f_2([a_0,a_1]) = x [a_0,a_1]$ we have $\partial_2^{ab} (f_2([a_0,a_1]))= f_1(\partial_2([a_0,a_1]))$,
that is $x \partial_2^{ab} [a_0,a_1] = f_1(\partial_2([a_0,a_1]))$. Applying $f_1$ to the formula
above we get
$$
f_1(\partial_2([a_0,a_1])) = 
\left( \sum_{k=0}^{\frac{e-2}{2}} (a_0a_1)^k \right)([a_0] - [a_1]+a_0[a_1]-a_1[a_0])
$$
so that $x = \sum_{k=0}^{\frac{e-2}{2}} (a_0a_1)^k$. Since $D_i(B) = 0$ for $i \geq 3$
this concludes the description of the morphism. In order to compute $H_2(f_i,\Z)$,
we apply the functor $\bullet \otimes_{\Z B} \Z$ to this morphism of complexes, and get the
following diagram

\begin{center}

\begin{tikzcd}
\dots \arrow[r] & 0 \arrow[r]\arrow[d] & \Z \arrow["0", r]\arrow["e/2", d] & \Z \oplus \Z \arrow["0",r]\arrow["\Id",d] & \Z \arrow["1", r]\arrow["1",d] & \Z \arrow[d, equal] \\
\dots \arrow[r] & 0 \arrow[r] & \Z \arrow["0",r] & \Z \oplus \Z \arrow["0",r] & \Z \arrow["1",r] & \Z  \\
\end{tikzcd}

\end{center}
so that the induced morphism $H_2(B,\Z) \to H_2(B^{ab},\Z)$
is multiplication by $e/2$. 

The Stallings-Stammbach exact sequence of \cite{STALLINGS,STAMMBACH} attached to the short exact sequence $1 \to C^2 B \to
B \to B^{ab} \to 1$ then reads
\begin{center}
\begin{tikzcd}
H_2(B,\Z) \arrow[r, "e/2"] \arrow[d, equal] & H_2(B^{ab},\Z)\arrow[d, equal] \arrow[r] &
\frac{C^2 B}{C^3 B}\arrow[r] & H_1(B,\Z) \arrow[d, equal] \arrow[r, "\simeq"] & H_1(B^{ab},\Z) \arrow[d,equal] \\
\Z & \Z & & \Z^2 & \Z^2 
\end{tikzcd}
\end{center}
so that $C^2 B/C^3 B \simeq \Z/(e/2)\Z$.
This concludes the proof of the Proposition.

\section{Proofs of the main results}

We can now prove Theorem \ref{theo1}. Since $\gr B$ is generated as a Lie algebra by $\gr_1 B = B^{ab}$,
we can take for generators the image of an arbitrary choice of elements $a_i \in S_i$,
and need to prove that $[a_i,a_j]= 0$ inside $\Q \otimes \gr_2 B$ for all $i,j$. Indeed, if we can do that,
then $\Q \otimes \gr B $ is a commutative Lie algebra over $\Q$ generated by $\gr_1 B$, so that
$\Q \otimes \gr B =\Q \otimes  \gr_1 B = \Q \otimes \gr B^{ab} $ and
$\gr_n B \otimes \Q = 0$ for every $n \geq 2$. Since $\gr B$ is generated by $\gr_1 B \simeq \Z^r$ it follows that each $\gr_n B$ is finitely generated as a $\Z$-module, hence
$\gr_n B$ is finite for each $n \geq 2$. But this implies for $n \geq 2$ that, for each $x \in TC^{n} G$,
we have $x^m \in C^{n} G$ for some $m \neq 0$, and then $(x^m)^N \in C^{n+1} G$ for $N = |\gr_n B|$, so that $x \in TC^{n+1} G$ and the sequence $TC^n G$ is stationnary. It follows
that 
$$
B = \varprojlim B/TC^n B \otimes \Q = B/TC^2 B \otimes \Q = B^{ab} \otimes \Q
$$
and this will prove Theorem \ref{theo1}.

So let us consider a pair $1 \leq i,j \leq r$ with $i\neq j$. By assumption, we have
that $e = m_{a_i,a_j}$ is even. Consider the subgraph $\Gamma_0$ of $\Gamma$
with vertices $a_i,a_j$. We have a natural homomorphism $B_0 = A(\Gamma_0) \to
A(\Gamma) = B$ mapping each $a_k$ to itself for $k=i,j$. It
induces a Lie algebra homomorphism $\Q \otimes \gr B_0  \to
\Q \otimes \gr B$. Since $[a_i,a_j] = 0$ inside $\Q \otimes \gr B_0$
by Proposition \ref{prop:dihedral}, we get that $[a_i,a_j] = 0$ inside $\Q \otimes \gr B$ and this concludes
the proof of Theorem \ref{theo1}.

For the proof of Theorem \ref{theo1b}, we use the isomorphism $A(\overline{\Gamma}) \otimes \Q
\simeq \exp \widehat{\mathcal{X}}(\overline{\Gamma})$ of Proposition \ref{prop:raags}.
Composing it with the natural morphism $A(\Gamma) \otimes \Q \to A(\overline{\Gamma}) \otimes \Q$
we get an homomorphism of pro-unipotent groups $\Phi : A(\Gamma)\otimes \Q \to \exp \widehat{\mathcal{X}}(\overline{\Gamma})$. In order to get the conclusion, we prove
that $\Phi$ is an isomorphism. For this it is sufficient to prove that the induced
morphism of Lie algebras $\varphi : \Lie (A(\Gamma) \otimes \Q) \to 
\widehat{\mathcal{X}}(\overline{\Gamma})$ is an isomorphism, and for this it is
sufficient to prove that the associated morphism between graded algebras
$\gr \varphi : \gr A(\Gamma) \otimes \Q \to \mathcal{X}(\overline{\Gamma})$ is an isomorphism.
We define a morphism of Lie algebras $\psi : \mathcal{X}(\overline{\Gamma}) \to \gr A(\Gamma)
\otimes \Q$ by mapping $x_{S_i}$ to the class of an arbitrary $s \in S_i$ inside
$\gr_1 A(\Gamma) \otimes \Q = A(\Gamma)^{ab}$. One needs to check for $i \neq j$ that, if
there exists $s \in S_i$ and $t \in S_j$ with $m_{s,t}< \infty$, then
$[s,t] = 0$ inside $\gr_1 A(\Gamma) \otimes \Q$, which we already checked using Proposition \ref{prop:dihedral}. Therefore $\psi$ is well-defined and $\psi\circ \gr \varphi$
is a Lie endomorphism of $\gr B \otimes \Q$ which maps each generator to itself.
It follows that $\psi \circ \gr \varphi$ is the identity, whence $\gr \varphi$
is injective. Since its image contains a generating set of $\mathcal{X}(\Gamma)$
it is surjective, so it is indeed an isomorphism, and this completes the proof of
Theorem \ref{theo1b}.

\bigskip

We finally prove Theorem \ref{theo2}, and refer to \cite{LEHRERTAYLOR} for general
results on irreducible complex reflection groups, including their Shephard-Todd classification into a general
series $G(de,e,n)$ depending on 3 integral parameters $d,e,n$ and the list $G_4,G_5,\dots G_{37}$
of exceptional groups.

 If $B \simeq \Z \times F_2$, then $B^{ab} \simeq \Z^3$, so that $c(W)=|\mathcal{A}/W| = 3$. Assume conversely that $c(W) = 3$. 
By the classification of irreducible complex reflection groups,
we have that either $W = G(de,e,2)$ for some $d > 1$ and $e$ even, or $W \in \{ G_7,G_{11},G_{19},
G_{15} \}$. In all these cases, it is known by  \cite{BANNAI} that the
corresponding braid group is isomorphic to $\Z \times F_2$, which is residually
torsion-free nilpotent, as all free groups are so.
In order to conclude the proof of Theorem \ref{theo2}, it thus remains to prove
that $B \otimes \Q \simeq B^{ab}\otimes\Q$ for all the other groups. 

If $W$ is a real
reflection group, this is known by Theorem \ref{theo1}. This is true
more generally for groups for which $B$ is an Artin group,
such as for instance the so-called Shephard groups studied in \cite{ORLIKSOLOMON}.
These Shephard groups cover
the exceptional groups $G_{25},G_{26},G_{32}$, as well as most of the
exceptional groups in dimension $2$. The remaining ones in dimension $2$ are then $G_{12},G_{13}$
and $G_{22}$. But the braid group of $G_{13}$ has been proved in \cite{BANNAI} to be
isomorphic to an Artin group, so this case is settled as well.

 The statement is also true for the groups
such that $c(W) = 1$, because in that case $gr B = \gr_1 B = \Z$, as $\gr B$
is generated as a Lie algebra by $\gr_1 B$. This covers $G_{12}$ and $G_{22}$, as well as the
remaining exceptional groups of rank at least $3$. This also covers the groups $G(e,e,n)$ for
$n \geq 3$ -- they are already covered when $n=2$, because in this case they
are real (dihedral) reflection groups in disguise.

Browsing the classification, the only
groups remaining to be considered are the groups $G(de,e,n)$ for $n \geq 3$
and $d > 1$. But in these cases, one gets immediately from the presentations
obtained in \cite{BMR} that the projection map $B \onto B^{ab} = \Z^2$ splits
(take the subgroup $\langle s, t_2 \rangle \simeq \Z^2$ from Table 1 there), so that
the induced map $H_2(B,\Z) \to H_2(B^{ab},\Z)$ admits a section and is therefore surjective.
Then, the Stallings-Stammbach exact sequence is
\begin{center}
\begin{tikzcd}
H_2(B,\Z) \arrow[r,  twoheadrightarrow]  & H_2(B^{ab},\Z) \arrow[r] &
\frac{C^2 B}{C^3 B}\arrow[r] & H_1(B,\Z) \arrow[r, "\simeq"] & H_1(B^{ab},\Z)  \\
\end{tikzcd}
\end{center}
so that $C^2 B/C^3 B = 0$ hence $\gr_2 B = 0$ and we conclude as before that $B \otimes \Q \simeq B^{ab}\otimes \Q$. This concludes the proof of the Theorem.

\end{document}